\documentclass{birkjour}
 \newtheorem{theorem}{Theorem}[section]
\newtheorem{lemma}[theorem]{Lemma}
\newtheorem{cor}[theorem]{Corollary}

\newtheorem{open question}{Open question}
\newtheorem{example}[theorem]{Example}
\newtheorem{program}{Program}
\theoremstyle{remark}
\newtheorem{remark}[theorem]{\bf{Remark}}
\newtheorem{note}[theorem]{Note}
 \numberwithin{equation}{section}
 \usepackage{enumerate}
\usepackage{lipsum}
\usepackage{float}
\usepackage{csquotes}
\usepackage{matlab-prettifier}
\usepackage[english]{babel} 
\usepackage{blindtext}
\usepackage{amsmath, amsthm, amscd, amsfonts, amssymb, graphicx, color}
\graphicspath{ {./images/} }
\usepackage[bookmarksnumbered, colorlinks, plainpages]{hyperref}
\hypersetup{colorlinks=true,linkcolor=red, anchorcolor=green, citecolor=cyan, urlcolor=red, filecolor=magenta, pdftoolbar=true}

\begin{document}

\title [On the numerical radius of a class of weighted shift operators ]{On the numerical radius of a class of weighted shift operators}

\author[Arobinda Ghosh]{Arobinda Ghosh}

\address{%
Department of Mathematics,\\ Jadavpur University, Kolkata 700032,\\ West Bengal, India
}
\email{7aghosh@gmail.com}

\author[Riddhick Birbonshi]{Riddhick Birbonshi}
\address{Department of Mathematics,\\ Jadavpur University, Kolkata 700032,\\ West Bengal, India}
\email{riddhick.math@gmail.com}

\author{Sarita Ojha}
\address{Department of Mathematics,\\
Indian Institute of Engineering Science and Technology, Shibpur,\\
Howrah-711103,\\
West Bengal, India}
\email{sarita.ojha89@gmail.com}


\subjclass{47A12, 47B37}
\keywords{Numerical Range, Numerical radius, Weighted shift operator}

\begin{abstract}
In this paper, we derive some bounds on numerical radius of the weighted shift operator $T$ with weights $(1,sq,q^2,tq^3,q^4,sq^5,q^6,tq^7,\ldots)$ where $s,t >0$ and $0<q<1$. Furthermore, we provide an entire function $F_T(z)$. The reciprocal of the minimal positive root of  $F_T(z)=0$ gives the numerical radius of $T$. These results generalize several previously known results on the numerical radius of weighted shift operators discussed in \cite{chakraborty2025numerical}.
\end{abstract}

\maketitle 
\footnotetext[1]{Corresponding author: Riddhick Birbonshi, Email: \texttt{riddhick.math@gmail.com}}
\section{Introduction and preliminaries}
Let $H$ be a complex separable Hilbert space with inner product $\langle \cdot, \cdot \rangle$ and norm $\|\cdot \|$. Let $B(H)$ denote the set of all bounded linear operators on $H$. For $T \in B(H)$, the numerical range $W(T)$ is the set $$W(T)=\{\langle Tx,x\rangle :x \in H, ||x||=1 \},$$
and the numerical radius $w(T)$ of $T$ is defined as $$w(T)=\sup\{|z|:z \in W(T)\}.$$
It is established that $W(T)$ is convex, bounded and nonempty subset of $\mathbb{C}$ (see \cite{gustafson1997numerical}). For an operator
$T\in B(H)$, the polar decomposition of $T$ is defined  as
$T=UP$, where 
$U$ is a partial isometry and 
$P$ is a positive operator (see \cite{halmos2012hilbert}). The operator $\Delta(T) = P^{1/2}UP^{1/2}$ is known as the Aluthge transformation of
$T$ (see \cite{aluthge1990p}).\\
Let $T$ be a weighted shift operator with bounded weights $(w_1,w_2,\ldots)$ on the Hilbert space $l^2 (\mathbb{N})$ represented  by the infinite matrix as,
\begin{eqnarray*}
T=T(w_1,w_2,\ldots) &=&  
\begin{bmatrix}
      0 & 0 & 0 & \cdots\\
      w_1 & 0 & 0 & \cdots\\
      0 & w_2 & 0 & \cdots\\
      \cdots & \cdots & \cdots & \cdots\\
      \cdots & \cdots & \cdots & \cdots
\end{bmatrix} . \end{eqnarray*} 
As a weighted shift operator
$T$ remains unitarily equivalent to $e^{i\theta}T$ for any real $\theta$, its numerical range is an open or closed circular disc centered at the origin (see \cite{shields1974weighted},\cite{stout1983numerical}). In particular, $W(T(1, 1, \ldots))$ is the open unit disc centered at the origin (see \cite{halmos2012hilbert}). \\
For the unitary operator 
\[
U=\operatorname{diag}(u_1,\, u_1u_2,\, u_1u_2u_3,\ldots),
\]
where $\{u_n:n=1,2,3,\ldots\}$ is a sequence of complex numbers satisfying $|u_n|=1$ for all $n$, the operator $UTU^{*}$ is a weighted shift with weights 
\[
(u_2w_1,\, u_3w_2,\, u_4w_3,\, \ldots,\, u_{n+1}w_n,\, \ldots).
\]
By choosing $u_1=1$, $u_{n+1}=\frac{\overline{w_n}}{|w_n|}$ when $w_n\neq 0$, and $u_{n+1}=1$ when $w_n=0$, we obtain 
\[
UTU^{*}=S,
\]
where $S = T(|w_1|, |w_2|, \ldots)$. Hence, without loss of generality, the weights of a weighted shift operator may be taken to be nonnegative.\\
The numerical radii of the following weighted shift operators with positive weights such as,
$$T(h,1,1, \ldots),\ T(a,b,a,b,\ldots),\ T(1,h,1,1,\ldots),$$ $$T(h,k,1,1,\ldots),\ T(w_1,w_2,\ldots,w_n,1,1,\ldots),\ T(h,k,a,b,a,b,\ldots),$$ $$T(w_1,w_2,\ldots ,w_{2n-1},b,a,b,a,\ldots),\ T(w_1,w_2,\ldots , w_{2n},a,b,a,b,\ldots)$$  
 have been calculated in \cite{berger1967mapping,ridge1976numerical,chien2013numerical,vandanjav2014numerical,undrakh2015numerical,chakraborty2022numerical,chakraborty20221numericall}, respectively.\\
In 1983, Stout \cite{stout1983numerical} has provided an algorithm to get the numerical radius of a Hilbert-Schmidt weighted shift operator $T(w_1,w_2, \ldots)$ by introducing the entire function 
\begin{equation}\label{F_T(z)}
    F_T(z)= \det \big(I-z \mbox{Re}(T(w_1,w_2, \ldots))\big) =1+\sum_{k=1} ^\infty \left(\frac{-1}{4}\right)^k c_kz^{2k},
\end{equation}
where Re$(T)=\frac{T+T^*}{2}$ and $c_k = \sum w_{i_1}^2w_{i_2}^2 \cdots w_{i_k}^2$, the sum being taken over 
$$1\leq i_1 < i_2 < \cdot \cdot \cdot <i_k< \infty ,\ i_2-i_1 \geq 2,\ \ldots ,\ i_k-i_{k-1} \geq 2.$$ 
Also, Stout \cite{stout1983numerical} has shown that the numerical radius $w(T(w_1,w_2,\ldots))=1/ \lambda$, where $\lambda$ is the smallest positive root of $F_T(z)=0$. \\
In $2009$, Chien and Nakazato \cite{chien2009numerical} have given an entire function 
$$ F_T (z) =1+\sum_{n=1} ^\infty \left(\frac{-1}{4}\right)^n\frac{ q^{2n(n-1)}}{(1-q^2)(1-q^4)\cdots(1-q^{2n})} z^{2n}, $$ 
for the weighted shift operator $T=T(1,q,q^2,\ldots)$ with $0<q<1$. Since computing the zeros of this entire function $F_T (z)$ is a difficult task, so they have provided  bounds on the numerical radius of the weighted shift operator $T$ as follows:
 \begin{eqnarray*}
        \frac{1}{4q^{3/2}}\sqrt{54-36q-2q^2-2\sqrt{(1-q)(9-q)^3}} &=& \sup_{0<\alpha<1} 
 \frac{\sqrt{\alpha}(1-\alpha)}{1-q\alpha} \\ &\leq& w\left(T(1,q,q^2, \ldots)\right)\leq \frac{1}{2-\sqrt{q}}.
    \end{eqnarray*}

Furthermore, Chakraborty and Ojha \cite{chakraborty2025numerical} generalized the above weighted shift operator as $T(1,sq,q^2, sq^3, \ldots)$ for $s>0, \ 0<q<1$ and they have proved the following result,
    \begin{eqnarray}\label{Bikshan_bound}
        \sup_{0<z<1} \frac{\sqrt{z}(1-z)(1+sqz)}{1-q^2z^2} &\leq& w\left(T(1,sq,q^2, sq^3, \ldots)\right)\nonumber\\ &&\leq \frac{\max\{1,sq\}}{2}+\frac{\sqrt{sq}}{2(2-\sqrt{q})}.
    \end{eqnarray}

The organization of the present article is as follows:\\
In Section \ref{2}, we focus on determining some bounds for the numerical radius of the weighted shift operator
\begin{equation}\label{Main_operator}
    T=T(1,sq,q^2,tq^3,q^4,sq^5,q^6,tq^7,\ldots), \ s,t>0,\ 0<q<1.
\end{equation}
In particular, for the lower bound of $w(T)$, we generalize the quantity $$\sup_{0<z<1} \frac{\sqrt{z}(1-z)(1+sqz)}{1-q^2z^2}$$ appearing in \eqref{Bikshan_bound} to $$\sup\limits_{0<z<1} \displaystyle\frac{(1-z)\sqrt z}{(1-q^4z^4)} \big((1+sqz)+q^2z^2(1+tqz)\big).$$
The computation of this supremum reduces to solving a polynomial equation of degree eight. Since an explicit solution of the equation can not be obtained for arbitrary values of $s$ and $t$, we pose this as an open question. In Section \ref{3}, we derive an entire function $F_T(z)$ corresponding to the weighted shift operator $T$, defined in  (\ref{Main_operator}). The exact numerical radius of $T$ is obtained from the least positive root of the equation $F_T(z)=0$. These results provide a generalization of earlier work established in \cite{chakraborty2025numerical}.

\section{Upper and lower bounds of the numerical radius of the operator\texorpdfstring{ $T(1,sq,q^2,tq^3,q^4,sq^5,q^6,tq^7,\ldots)$}{T(1,sq,q²,tq³,q⁴,sq⁵,q⁶,tq⁷,...)}}\label{2}
Here, we obtain some bounds on numerical radius of the weighted shift operator $T=T(1,sq,q^2,tq^3,q^4,sq^5,q^6,tq^7,\ldots)$ where $s>0,\ t>0$ and $0<q<1$.
\begin{lemma}\label{HS}
    Let $T=T(1,sq,q^2,tq^3,q^4,sq^5,q^6,tq^7,\ldots)$ be a weighted shift operator with $s>0,\ t>0$ and $0<q<1$. Then $T$ is a Hilbert-Schmidt operator.
\end{lemma}
    \begin{proof}
        Let $\{e_k:k\in\mathbb{N}\}$ be the standard orthonormal basis of $\ell^2(\mathbb{N})$. Then the Hilbert-Schmidt norm of $T$ is 
\begin{eqnarray*}
        ||T||_{HS} ^2 &=& \sum_{k=1} ^\infty ||Te_k||^2\\ 
        &=& 1+s^2q^2+q^4+t^2q^6+q^8+s^2q^{10}+q^{12}+t^2q^{14}+\cdots \\ 
        &=& (1+q^4+q^8+\cdots)+s^2q^2(1+q^8+q^{16}+\cdots)\\ &&+t^2q^6(1+q^8+q^{16}+\cdots)\\ 
        &=& \frac{1}{1-q^4}+ \frac{q^2(s^2+t^2q^4)}{1-q^8} < \infty .
\end{eqnarray*}
        Hence, $T$ is a Hilbert-Schmidt operator.
    \end{proof}
From Lemma \ref{HS}, the operator $T=T(1,sq,q^2,tq^3,q^4,sq^5,q^6,tq^7,\ldots)$ is a compact operator. Hence $W_e(T)=\{0\} \subset W(T)$ and therefore by \cite{lancaster1975boundary}, $W(T)$ is closed.

\begin{theorem}\label{ag}
    Let $T=T(1,sq,q^2,tq^3,q^4,sq^5,q^6,tq^7,\ldots)$ be a weighted shift operator with $s>0,\ t>0$ and $0<q<1$. Then \begin{enumerate}[\upshape (a)]
        \item $w(T)\geq \sup\limits_{0<z<1} \displaystyle\frac{(1-z)\sqrt z}{(1-q^4z^4)} \big((1+sqz)+q^2z^2(1+tqz)\big)$
        \item $w(T)\leq \frac{1}{2}\max\{1,sq,tq^3\} + \displaystyle\frac{\sqrt{sq}}{2}w\big( 
  T(1,q,\sqrt{\frac{t}{s}}q^2,\sqrt{\frac{t}{s}}q^3,q^4,q^5,\ldots)\big)$.
    \end{enumerate}
\end{theorem}
    \begin{proof}
    \begin{enumerate}
    \item[(a)] 
Consider the unit vector $x=\{x_n \} \in l^2 (\mathbb{N})$, where $x_n= \sqrt{ (1-z)} z^\frac{(n-1)}{2}$, $0<z<1.$  Then
         \begin{eqnarray*}
     \langle Tx,x \rangle
        &=& x_1x_2 +sqx_2x_3 +q^2x_3x_4 +tq^3x_4x_5+\cdots \\
         &=& (1-z)z^{1/2} +sq(1-z) z^{3/2}+ q^2(1-z)z^{5/2} +tq^3(1-z)z^{7/2}+\cdots\\
         &=& (1-z)z^{1/2} \big \{(1+sqz)(1+q^4z^4+q^8z^8+\cdots)\\
         &&+(1+tqz)(q^2z^2+q^6z^6+\cdots) \big\} \\
         &=& \frac{(1-z)\sqrt z}{(1-q^4z^4)}\left( (1+sqz)+q^2z^2(1+tqz)\right) >0.
     \end{eqnarray*}
Hence the result follows. 
 \item[(b)] We know from \cite{yamazaki2007upper},
 $$w(T)\leq \frac{\|T\|}{2}+\frac{w (\Delta(T))}{2}.$$
Now the polar decomposition of $T$ is $T=UP$, where  $$U=T(1,1,1,\ldots)\ \mbox{ and } \ P= \sqrt{T^*T}=diag(1,sq,q^2,tq^3,\ldots)$$
and therefore
\begin{eqnarray*}
\Delta(T) &=& P^{1/2}UP^{1/2}\\ &=& \sqrt{sq} \  
  T(1,q,\sqrt{\frac{t}{s}}q^2,\sqrt{\frac{t}{s}}q^3,q^4,q^5,\sqrt{\frac{t}{s}}q^6,\sqrt{\frac{t}{s}}q^7,q^8,q^9,\ldots).
\end{eqnarray*}
Also $||T|| =\max \{1,sq,tq^3 \} $.
Hence we get the required upper bound.
\end{enumerate}
\end{proof}
Again using 
\begin{equation*}
    \left\|T(1,q,\sqrt{\frac{t}{s}}q^2,\sqrt{\frac{t}{s}}q^3,q^4,q^5,\sqrt{\frac{t}{s}}q^6,\sqrt{\frac{t}{s}}q^7,q^8,q^9,\ldots)\right\|=\max \left\{1, \sqrt{\frac{t}{s}}q^2\right\},
\end{equation*}
 we have the following corollary.

\begin{cor}\label{cor}
    Let $T=T(1,sq,q^2,tq^3,q^4,sq^5,q^6,tq^7,\ldots)$ be a weighted shift operator with $s>0, \ t>0$ and $0<q<1$. Then 
    $$w(T)\leq \frac{1}{2}\max\{1,sq,tq^3\}+\frac{1}{2}\sqrt{sq} \max \left\{1, \sqrt{\frac{t}{s}}q^2\right\}.$$
\end{cor}

\begin{remark}\label{remark}
Several notable special cases from our operator are as follows:
\begin{enumerate}
    \item For $s=t=1$, our operator reduces to $T(1,q,q^2, \ldots)$ and Theorem \ref{ag} coincides with Theorem 2.1 of \cite{chien2009numerical}.
    \item For $s=t$, our operator becomes the operator $T(1,sq,q^2, sq^3, \ldots)$ and Theorem \ref{ag} correspondingly reduces to Theorem 2.3 of \cite{chakraborty2025numerical}.
\end{enumerate}
\end{remark}

\noindent To calculate the lower bound, we need to determine
$$\sup_{0<z<1} \frac{(1-z)z^{1/2}\bigg(1+sqz+ q^2z^2(1+tqz)\bigg)}{(1-q^4z^4)}.$$
Consider the function 
\begin{equation}\label{f(z)}
    f(z)=\frac{(1-z)z^{1/2}\bigg(1+sqz+ q^2z^2(1+tqz)\bigg)}{(1-q^4z^4)}
\end{equation}
defined on $0 \leq z \leq 1.$

Since $f$ is a nonnegative real-valued continuous function on $[0,1]$ and $f(0)=f(1)=0$, so $f$ attains its maximum in (0,1). To find the extreme points of the function $f(z)$ in $(0,1)$, we first obtain
         \begin{eqnarray*}
             f'(z)&=&\frac{1}{2\sqrt{z}(1-q^4z^4)^2}\bigg(tq^7z^8 +q^6(tq-1)z^7 +3q^5(q-s)z^6+ 5q^4(sq-1)z^5 \\&&+q^3(7q-9t)z^4+7q^2(tq-1)z^3
         +5q(q-s)z^2+3(sq-1)z +1\bigg).
         \end{eqnarray*}
So, the extreme points of $f(z)$ in $(0,1)$ are the roots of $g(z)=0$ that lies in $(0,1)$ where 
\begin{eqnarray}\label{g(z)}
    g(z) &=& tq^7z^8 +q^6(tq-1)z^7 +3q^5(q-s)z^6+ 
  5q^4(sq-1)z^5 +q^3(7q-9t)z^4\nonumber \\ && +7q^2(tq-1)z^3+5q(q-s)z^2+3(sq-1)z +1.
\end{eqnarray}
It is easy to observe that
\begin{eqnarray*}
    g(0) &=& 1>0,\\
     \mbox{and }g(1) &=& 2(q^4-1)(tq^3+q^2+sq+1)<0.
\end{eqnarray*}
So $g(z)=0$ has a root in $(0,1)$.

\noindent To solve the equation $g(z)=0$ for $0<q<1$, we divide the cases as follows:
         \begin{enumerate}
            \item[] Case 1: $s,t \in (q , \frac{1}{q}) $ 
            \item[] Case 2: at least one of $s, t \in (0,q]$
            \item[] Case 3: at least one of $s,t\in [\frac{1}{q}, \infty)$.
         \end{enumerate}
\noindent To prove the results, we start with Case 1.
\begin{theorem}\label{solved}
     Let $s,t\in (q, \frac{1}{q})$ , then $g(z)=0$ in \eqref{g(z)} has a unique root in $(0,1).$
\end{theorem}
\begin{proof}
In this case, $s,t>q$ and $sq-1, tq-1<0$. Therefore by Descartes' rule of signs, $g(z)=0$ has either 0 or 2 positive roots. As $g(0)g(1)<0$, let us assume $g(z)=0$ has two distinct roots in $(0,1)$. So there exists $c\in (0,1)$ such that $g'(c)=0$. Now
\begin{eqnarray*}
   g'(z) &=& 8tq^7z^7+7q^6(tq-1)z^6+18q^5(q-s)z^5+25q^4(sq-1)z^4\\
        && +4q^3(7q-9t)z^3+21q^2(tq-1)z^2+10q(q-s)z+3(sq-1). 
\end{eqnarray*}
Since $g'(0)=3(sq-1)<0$ and $g'(1)=(q^4-1)(15tq^3+11q^2+7sq+3)<0$, therefore, $g'(z)=0$ has either no roots or an even number of  roots in $(0,1)$. Also, by Descartes' rule of signs, $g'(z)=0$ has only one positive root.\\   
Hence, $g'(z)=0$ has no roots in $(0,1)$, which is a contradiction. So, $g(z)=0$ has only one root in $(0,1)$. 
\end{proof}

Here, we present two examples illustrating cases where the above result holds. In these examples, the numerical radius achieves a better lower bound than the classical $\frac{\|T\|}{2}$, 
and the upper bound is sharper than the traditional norm bound $\|T\|$.
\begin{example}
         Consider $q=\frac{1}{5}$, $s=4.8, \ t=1$. Then $s,t\in (q, \frac{1}{q})$ with $s>t$ and
         \begin{eqnarray*}
             g(z) &=& 0.0000128z^8
         -0.0000512z^7-0.004416z^6-0.00032z^5\\ &&-0.0608z^4-0.224z^3-4.6z^2-0.12z+1.
         \end{eqnarray*}
         The approximate roots of $g(z)=0$ are
         \begin{align*}
            & -17.0917,\  -0.4846,\ 0.4481,\ 21.0930, \ -3.4438 \pm 4.1343i 
         \\ \mbox{and } & 3.4615 \pm 4.7422i.
         \end{align*}
    So, the smallest positive root  in the interval $(0,1)$ is $\alpha = 0.4481$, and $f(\alpha)= 0.5316$, where $f(z)$ is defined in (\ref{f(z)}). Hence $$\frac{||T||}{2}=0.5<0.5316\leq w(T).$$
Also from Corollary \ref{cor}, 
\begin{align*}
    w(T)\leq \frac{1}{2}\max\{1,sq,tq^3\}+\frac{1}{2}\sqrt{sq} \max \{1, \sqrt{\frac{t}{s}}q^2\}=0.9899<1=||T||.
\end{align*}\end{example}

\begin{example}
          Consider $q=0.7$, $s=q^{1/2}, t=q^{1/4}$. Then $s,t\in (q,\frac{1}{q})$ with $s<t$ and
          \begin{eqnarray*}
              g(z)&=&0.075z^8-0.0423z^7-0.0685z^6-0.497z^5-1.14z^4\\&&-1.235z^3-0.476z^2-1.244z+1.
          \end{eqnarray*}
    The approximate roots of $g(z)=0$ are 
    \begin{align*}
        &0.5049,\ 2.7856,\ 0.1177 \pm  1.1638i,\ -0.1224 \pm 1.7352i\\ \mbox{and}& -1.3585 \pm 0.6664i.
    \end{align*}
    So, the smallest positive root in the interval $(0,1)$ is $\alpha = 0.5049$, and $f(\alpha)= 0.5221$, where $f(z)$ is defined in (\ref{f(z)}). Hence $$\frac{||T||}{2}=0.5<0.5221\leq w(T).$$
    Also from Corollary \ref{cor}, 
\begin{align*}
    w(T)\leq \frac{1}{2}\max\{1,sq,tq^3\}+\frac{1}{2}\sqrt{sq} \max \{1, \sqrt{\frac{t}{s}}q^2\}=0.8826<1=||T||.
\end{align*}
\end{example}

Recently, Chakraborty and Ojha \cite{chakraborty2025numerical} investigated the operator 
\begin{equation}\label{T_1}
    T_1=T(1,sq,q^2,sq^3,\ldots) \mbox{ for }s>0, \ 0<q<1
\end{equation} which is a special case of our operator \eqref{Main_operator} for $t=s$ (see Remark \ref{remark}). Taking $s = t$, the expression $g(z)$ in \eqref{g(z)} shows that
\begin{eqnarray*}
    g(z) &=& sq^7z^8 +q^6(sq-1)z^7 +3q^5(q-s)z^6+ 
  5q^4(sq-1)z^5 +q^3(7q-9s)z^4\nonumber \\ && +7q^2(sq-1)z^3
         +5q(q-s)z^2+3(sq-1)z +1 \\&=& (q^2z^2+1)^2\bigg(sq^3z^4+q^2(sq-1)z^3+(3q^2-5sq)z^2+3(sq-1)z+1\bigg)\\&=&(q^2z^2+1)^2 h(z)
\end{eqnarray*}
where
\begin{equation}\label{h(z)}
    h(z)=sq^3z^4+q^2(sq-1)z^3+(3q^2-5sq)z^2+3(sq-1)z+1
\end{equation}
is the polynomial given in equation (7) of \cite{chakraborty2025numerical}. Consequently, the real roots of $g(z)$ and $h(z)$ coincide when $s = t$. To obtain the lower bound of $w(T_1)$, the authors have considered the function $h(z)$ as defined in \eqref{h(z)}. They have shown that $h(z)$ has a unique root in the interval $\left(0, 1\right)$. Particularly, for the case $s \in (0, q)$, they used the monotone decreasing behavior of $h(z)$ on $\left(\tfrac{1}{3}, 1\right)$ to establish the existence of a unique root in $\left(\tfrac{1}{3}, 1\right) \subset (0,1)$.

In contrast, for our operator \eqref{Main_operator}, the corresponding function $g(z)$ exhibits both increasing and decreasing behavior on $\left(\tfrac{1}{3}, 1\right)$ whenever at least one of $s,t \in (0,q]$ or at least one of $s,t \in \left[\tfrac{1}{q}, \infty\right)$. To illustrate this claim, we provide the following examples. 
 \begin{example}
     Consider $q=0.995$, $s=0.01$, $t=0.966$, i.e., $s,t \in (0,q]$. The curve of 
     \begin{eqnarray*}
         g(z) &=& 0.932z^8 -0.0376z^7+2.881z^6-4.8519z^5-1.703z^4-0.269z^3\\
         &&+4.9003z^2-2.97z+1
     \end{eqnarray*}
      exhibits both increasing and decreasing behavior in the interval $(\frac{1}{3}, 1)$ as in Figure \ref{fig:mesh1}. However, $g(z)=0$  has a unique root in $(\frac{1}{3},1)$, which is approximately $0.9160$.   
 \begin{figure}[H]
     \centering
 \includegraphics[width=0.4\textwidth]{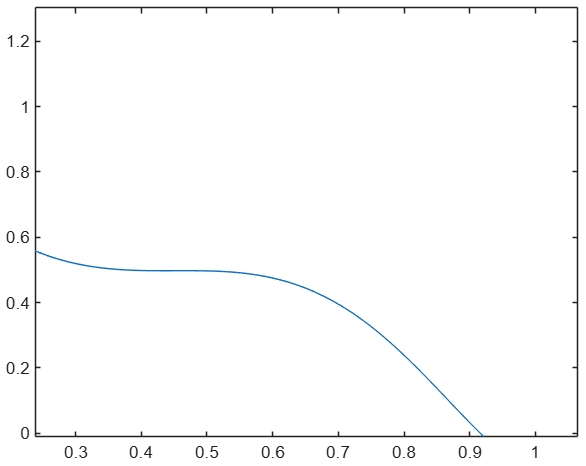}
     \caption{}
     \label{fig:mesh1}
 \end{figure}
 \end{example}

\begin{example}
    Consider $q=0.99$, $s=1.2$, $t=3$, i.e.,  $s,t \in [\frac{1}{q},\infty)$. The curve of 
    \begin{eqnarray*}
   g(z) &=& 2.796196z^8+1.854715894z^7-0.599123z^6+0.90296z^5\\&&-19.4739z^4+13.515579z^3-1.0395z^2+0.564z+1 
    \end{eqnarray*}
    exhibits both increasing and decreasing behavior in $(\frac{1}{3}, 1)$ as in Figure \ref{fig:mesh2}. However, $g(z)=0$  has a unique root in $(\frac{1}{3},1)$, which is approximately $0.9070$.
\begin{figure}[h!]
    \centering
\includegraphics[width=0.4\textwidth]{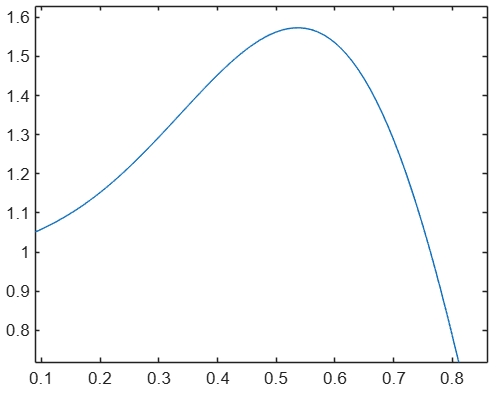}
    \caption{}
    \label{fig:mesh2}
\end{figure}
\end{example}

\begin{example}
     Consider $q=0.977$, $s=0.725$, $t=1.53$, i.e., $s\in (0,q]$ and $t\in [\frac{1}{q},\infty)$. The curve of\begin{eqnarray*}
g(z)&=&1.295074z^8+0.429021z^7+0.671577z^6-1.326708z^5\\&&-6.460z^4+3.305z^3+1.231z^2-0.875025z+1 \end{eqnarray*}  
exhibits both increasing and decreasing behavior in $(\frac{1}{3}, 1)$ as in Figure \ref{fig:mesh3}. However, $g(z)=0$  has a unique root in $(\frac{1}{3},1)$, which is approximately $0.8411$.
\begin{figure}[h!]
    \centering
\includegraphics[width=0.4\textwidth]{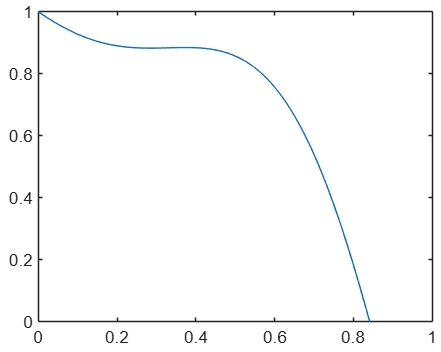}
    \caption{}
    \label{fig:mesh3}
\end{figure}
\end{example}

It is also observed that
\begin{align*}
    g\left(\frac{1}{3}\right) = \frac{4}{9}sq+\frac{8}{27}q^2+\frac{4}{27}tq^3+\frac{16}{3^5}q^4+\frac{4}{3^5}sq^5+\frac{8}{3^7}q^6+\frac{4}{3^8}tq^7>0.
\end{align*}
Thus $g(z)=0$ in \eqref{g(z)} has a root in $(\frac{1}{3},1) \subset (0,1)$.\\
The polynomial $g(z)$ can now be rewritten as 
\begin{eqnarray}\label{G(z)}
   g(z)&=&1-3z+qzm(z,s,t,q),
\end{eqnarray}
where \begin{eqnarray}\label{m_z}
    m(z,s,t,q) &=&-7qz^2+ 3s+5qz-5sz+7tq^2z^2 +7q^3z^3-9tq^2z^3+5sq^4z^4\nonumber \\
 &&-5q^3z^4+3q^5z ^5-3sq^4z^5+tq^6z ^6-q^5z^6+tq^6z^7. 
 \end{eqnarray}
To prove the next result, the following lemma is essential.

\begin{lemma}\label{MZ}
Let $z\in(0,\frac{1}{3})$. Then $m(z,s,t,q)$ defined in \eqref{m_z} is positive for all $0<z<\frac{1}{3}$.
\begin{proof}
    From \eqref{m_z}, \begin{eqnarray*}
        m(z,s,t,q)&=& -7qz^2+5qz+7q^3z^3-5q^3z^4+3q^5z ^5-q^5z^6\nonumber\\&&+s(3-5z+5q^4z^4-3q^4z^5)+t(7q^2z^2-9q^2z^3+q^6z^6+q^6z^7)\nonumber\\
 &=& R_1(q,z)+sR_2(q,z)+tR_3(q,z),
    \end{eqnarray*}
    where \begin{eqnarray*}
        R_1(q,z)&=&-7qz^2+5qz+7q^3z^3-5q^3z^4+3q^5z ^5-q^5z^6,\\
            R_2(q,z)&=& 3-5z+5q^4z^4-3q^4z^5\\
         \mbox{and }  \ R_3(q,z)&=& 7q^2z^2-9q^2z^3+q^6z^6+q^6z^7.
        \end{eqnarray*}
    For $0<z<\frac{1}{3}$, 
    \begin{eqnarray*}
        R_1(q,z)&=&qz(5-7z)+q^3z^3(7-5z)+q^5z ^5(3-z)\\&>& \frac{8qz}{3}+\frac{16q^3z^3}{3}+\frac{8q^5z^5}{3} > 0.
    \end{eqnarray*} 
    Similarly, 
    \begin{eqnarray*}
            R_2(q,z) &=& 3-5z+q^4z^4(5-3z)> \frac{4}{3}+4q^4z^4 > 0\\
        \mbox{and } 
            R_3(q,z) &=& 7q^2z^2-9q^2z^3+q^6z^6+q^6z^7 > q^2z^2(7-9z)> 4q^2z^2>0.
        \end{eqnarray*} Hence the result follows.
\end{proof}
\begin{theorem}
 The equation $g(z)=0$ has no root in $(0,\frac{1}{3})$, where $g(z)$ is defined in \eqref{G(z)}.
\end{theorem}
\begin{proof}
    Let $z_0$ be a root of $g(z)=0$ in $(0,\frac{1}{3})$. Then from \eqref{G(z)}, we have $m(z_0,s,t,q)=\frac{3z_0-1}{qz_0}<0$ which is a contradiction by Lemma \ref{MZ}. Hence the result follows.
\end{proof}
\end{lemma}
Thus, it is clear that no root of $g(z)=0$ lies in the interval $(0, \frac{1}{3}]$, while at least one root must occur in $(\frac{1}{3}, 1)$. The uniqueness of this root is established only in Case~1 (see Theorem \ref{solved}). For the remaining cases, i.e., when at least one of $s,t\notin [q,\frac{1}{q}]$, this question remains unresolved. Motivated by extensive examples, the following question naturally arises.
\begin{open question}
    For $s,t>0$ and $0<q<1$ does the polynomial $g(z)$ defined in \eqref{g(z)} have a unique root in $(\frac{1}{3},1)$?
\end{open question}

\section{Formulate the numerical radius of weighted shift operator }\label{3}
In this section, we obtain the coefficients $\{c_n\}$ of entire function $F_T(z)$, as defined in \eqref{F_T(z)}, for the weighted shift operator defined in \eqref{Main_operator}. The smallest positive solution of 
$F_T(z)=0$ yields the precise numerical radius of the operator $T$. For this, consider the following weighted shift operators,
\begin{eqnarray}
    T_2 &=& T(s,q,tq^2,q^3,\ldots,sq^{2n-2},q^{2n-1},tq^{2n},q^{2n+1},\ldots) \label{T_2}\\
    T_3 &=& T(t,q,sq^2,q^3,\ldots,tq^{2n-2},q^{2n-1},sq^{2n},q^{2n+1},\ldots)\label{T_3}\\
    T_4 &=& T(1,tq,q^2,sq^3,\ldots,q^{2n-2},tq^{2n-1},q^{2n},sq^{2n+1},\ldots). \label{T_4}
\end{eqnarray}
If we take $T_2, T_3, T_4$ in place of $T$ in \eqref{F_T(z)} then $c_k$ changes to 
\begin{eqnarray}\label{d_k,h_k,e_k}
    d_k=\sum u_{i_1} ^2\cdots u_{i_k} ^2,\ h_k=\sum p_{i_1} ^2\cdots p_{i_k} ^2 \mbox{ and } e_k=\sum v_{i_1} ^2\cdots v_{i_k} ^2\ ,
\end{eqnarray} 
respectively, all the the sums being taken over $$1\leq i_1<i_2 < \cdots <i_k <\infty ,\ i_2-i_1\geq 2,\ i_3-i_2\geq 2,\ \cdots,\ i_k-i_{k-1}\geq 2,$$
where  
\begin{eqnarray*}
     u_{2k} &=& q^{2k-1} \mbox{ for } k\geq 1 \mbox{ and } u_{2k-1} = \begin{cases}
         sq^{2k-2} & \mbox{for } k=1,3,5,\ldots\\
        tq^{2k-2}&  \mbox{for } k=2,4,6,\ldots,
    \end{cases}\\
    p_{2k} &=& q^{2k-1} \mbox{ for } k\geq 1 \mbox{ and } p_{2k-1} = \begin{cases}
         tq^{2k-2}& \mbox{for } k=1,3,5,\ldots\\
        sq^{2k-2}&  \mbox{for } k=2,4,6,\ldots,
    \end{cases}\\
    v_{2k-1} &=& q^{2k-2} \mbox{ for } k\geq 1 \mbox{ and } v_{2k} = \begin{cases}
         tq^{2k-1}& \mbox{for } k=1,3,5,\ldots\\
        sq^{2k-1}&  \mbox{for } k=2,4,6,\ldots
    \end{cases} \ \ .
\end{eqnarray*}
We define 
\begin{eqnarray*}
\underbrace{(t^2\ldots t^2)(s^2\ldots s^2)}_{n \text{ terms}}=\begin{cases}
         (t^2)^\frac{n+1}{2}(s^2)^{\frac{n-1}{2}} & \mbox{for $n$ is odd }  \\
       (t^2)^\frac{n}{2}(s^2)^\frac{n}{2} &  \mbox{for $n$ is even } 
    \end{cases}
 \end{eqnarray*}   
    and
\begin{eqnarray*}
\underbrace{(s^2\ldots s^2)(t^2\ldots t^2)}_{n \text{ terms}}=\begin{cases}
         (s^2)^\frac{n+1}{2}(t^2)^{\frac{n-1}{2}} & \mbox{for $n$ is odd }  \\
       (s^2)^\frac{n}{2}(t^2)^\frac{n}{2} &  \mbox{for $n$ is even }.
    \end{cases}
 \end{eqnarray*}   
Now, for $k=1$, from \eqref{d_k,h_k,e_k} we obtain 
\begin{eqnarray*}
           d_1 &=&\sum_{k=1} ^\infty u_k ^2= s^2+q^2 +(tq^2)^2+(q^3)^2+\cdots=
           \frac{(s^2+t^2q^4)}{(1-q^8)}+ \frac{q^2}{(1-q^4)},\\
            h_1 &=&\sum_{k=1} ^\infty p_k ^2= t^2+q^2 +(sq^2)^2+(q^3)^2+\cdots= 
           \frac{(t^2+s^2q^4)}{(1-q^8)}+ \frac{q^2}{(1-q^4)},\\
           e_1 &=&\sum_{k=1} ^\infty v_k ^2= 1+t^2q^2 +(q^2)^2+(sq^3)^2+\cdots= 
           \frac{(t^2+s^2q^4)q^2}{(1-q^8)}+ \frac{1}{(1-q^4)}.           
       \end{eqnarray*}
For $n\geq1$, we have the following lemma.
\begin{lemma}\label{d_{n+1},h_{n+1},e_{n+1}}
Let the weighted shift operators $T_2, \ T_3,\ T_4$ be defined as in \eqref{T_2}, \eqref{T_3}, \eqref{T_4} respectively. Then  for $n\geq 1$, the following recurrence relations hold:
     \begin{eqnarray}
     d_{n+1}&=& \frac{q^{4n}}{1-q^{8n+8}}\big(s^2h_n+t^2q^{4n+4}d_n+q^{2n+2}c_n+q^{6n+6}e_n \big),\label{dn+1}\\
        h_{n+1} &=& \frac{q^{4n}}{1-q^{8n+8}}\big(t^2d_n+q^{2n+2}e_n+s^2q^{4n+4}h_n+q^{6n+6}c_n \big)\label{hn+1} 
        \end{eqnarray}
        and
       \begin{eqnarray} 
      e_{n+1}&=& \frac{q^{4n}}{1-q^{8n+8}}\big(c_n+q^{4n+4}e_n+t^2q^{2n+2}d_n+s^2q^{6n+6}h_n \big) \label{en+1},
     \end{eqnarray}
   where $\{d_n\},\ \{h_n\},\ \{e_n\}$ are defined in \eqref{d_k,h_k,e_k} and 
    $\{c_n\}$ be the coefficients of the entire function $F_T(z)$ defined in \eqref{F_T(z)} for the operator \eqref{Main_operator}.
\end{lemma}
\begin{proof}
   Here
   \begin{align*}
       d_{n+1}=& \ s^2 \bigg \{ \sum_{3\leq j_1,j_1+2\leq j_2,\ldots,j_{n-1}+2 \leq j_n} u_{j_1} ^2 u_{j_2} ^2\cdots u_{j_n} ^2\bigg\}\\ 
       & +q^2\bigg\{ \sum_{4\leq j_1,j_1+2\leq j_2,\ldots ,j_{n-1}+2 \leq j_n} u_{j_1} ^2 u_{j_2} ^2\cdots u_{j_n} ^2\bigg\} \\
       &+ t^2q^4\bigg\{\sum_{5\leq j_1,j_1+2\leq j_2,\ldots, j_{n-1}+2 \leq j_n} u_{j_1} ^2 u_{j_2} ^2\cdots u_{j_n} ^2 \bigg\}\\
       &+q^6\bigg\{\sum_{6\leq j_1,j_1+2\leq j_2,\ldots, j_{n-1}+2 \leq j_n} u_{j_1} ^2 u_{j_2} ^2\cdots u_{j_n} ^2 \bigg\} +\cdots 
   \end{align*}
Putting the values of $\{u_k\}$ as defined above, we get
\begin{eqnarray}
    d_{n+1}\nonumber&=&
s^2q^{4n}\bigg(\underbrace{(t^2\ldots t^2)(s^2\ldots s^2)}_{n \text{ terms}}\times 1\times q^4 q^8\ldots q^{4(n-1)}+ q^2q^6q^{10}\ldots q^{4n-2}+ \ldots \bigg) \nonumber\\ 
&&+ q^2 q^{6n}\bigg(1\times q^4q^8\ldots q^{4(n-1)}+ \underbrace{(s^2\ldots s^2)(t^2\ldots t^2)}_{n \text{ terms}} \times q^2q^6 q^{10} \ldots q^{4n-2}+\cdots \bigg)\nonumber \\
&& +t^2q^4 q^{8n}\bigg(\underbrace{(s^2\ldots s^2)(t^2\ldots t^2)}_{n \text{ terms}}\times 1\times q^4 q^8\ldots q^{4(n-1)}+ q^2q^6q^{10}\ldots q^{4n-2}+ \ldots \bigg)\nonumber\\
&&+ q^6 q^{10n}\bigg( 1\times q^4q^8\ldots q^{4(n-1)}+\underbrace{(t^2\ldots t^2)(s^2\ldots s^2)}_{n \text{ terms}} \times q^2q^6 q^{10} \ldots q^{4n-2}+\cdots \bigg)\nonumber\\
&&+\cdots\nonumber\\
&=&
s^2(q^{4n}+q^{12n+8}+\cdots)h_n + (q^{6n+2}+q^{14n+10}+\cdots)c_n \nonumber \\
&&+ t^2(q^{8n+4}+q^{16n+12}+\cdots)d_n + (q^{10n+6}+q^{18n+14}+\cdots)e_n
 \nonumber\\       
 &=&\frac{q^{4n}}{1-q^{8n+8}}\big(s^2h_n+t^2q^{4n+4}d_n+q^{2n+2}c_n+q^{6n+6}e_n \big).\nonumber
\end{eqnarray}
Again 
\begin{eqnarray}
    h_{n+1}&=&
    t^2 \bigg \{ \sum_{3\leq j_1,j_1+2\leq j_2,\ldots,j_{n-1}+2 \leq j_n} p_{j_1} ^2 p_{j_2} ^2\cdots p_{j_n} ^2\bigg\} \nonumber   \\
        &&+q^2\bigg\{ \sum_{4\leq j_1,j_1+2\leq j_2,\ldots ,j_{n-1}+2 \leq j_n} p_{j_1} ^2 p_{j_2} ^2\cdots p_{j_n} ^2\bigg\}\nonumber \\&&+ s^2q^4\bigg\{\sum_{5\leq j_1,j_1+2\leq j_2,\ldots, j_{n-1}+2 \leq j_n} p_{j_1} ^2 p_{j_2} ^2\cdots p_{j_n} ^2 \bigg\}\nonumber \\&&+q^6\bigg\{\sum_{6\leq j_1,j_1+2\leq j_2,\ldots, j_{n-1}+2 \leq j_n} p_{j_1} ^2 p_{j_2} ^2\cdots p_{j_n} ^2 \bigg\} +\cdots\nonumber
    \end{eqnarray}
    Putting the values of $\{p_k\}$ as defined above, we get
    \begin{eqnarray}
     h_{n+1}   &=&
t^2q^{4n}\bigg(\underbrace{(s^2\ldots s^2)(t^2\ldots t^2)}_{n \text{ terms}}\times 1\times q^4 q^8\ldots q^{4(n-1)}+ q^2q^6q^{10}\ldots q^{4n-2}+ \cdots \bigg)\nonumber \\
&&+ q^2 q^{6n}\bigg(1\times q^4q^8\ldots q^{4(n-1)}+\underbrace{(t^2\ldots t^2)(s^2\ldots s^2)}_{n \text{ terms}} \times q^2q^6 q^{10} \ldots q^{4n-2}+\cdots \bigg)\nonumber \\
&& +s^2q^4 q^{8n}\bigg(\underbrace{(t^2\ldots t^2)(s^2\ldots s^2)}_{n \text{ terms}}\times 1\times q^4 q^8\ldots q^{4(n-1)}+ q^2q^6q^{10}\ldots q^{4n-2}+ \cdots \bigg)\nonumber \\
&&+ q^6 q^{10n}\bigg( 1\times q^4q^8\ldots q^{4(n-1)}+\underbrace{(s^2\ldots s^2)(t^2\ldots t^2)}_{n \text{ terms}} \times q^2q^6 q^{10} \ldots q^{4n-2}+\cdots \bigg)\nonumber\\
&&+\cdots \nonumber
\end{eqnarray}
\begin{eqnarray}
&=& t^2(q^{4n}+q^{12n+8}+\cdots)d_n + (q^{6n+2}+q^{14n+10}+\cdots)e_n \nonumber \\
&& +s^2(q^{8n+4}+q^{16n+12}+\cdots)h_n + (q^{10n+6}+q^{18n+14}+\cdots)c_n\nonumber\\ 
&=&\frac{q^{4n}}{1-q^{8n+8}}\big(t^2d_n+q^{2n+2}e_n+s^2q^{4n+4}h_n+q^{6n+6}c_n \big).\nonumber
\end{eqnarray}
Also
\begin{eqnarray}
     e_{n+1}
     &=&
     \sum_{3\leq j_1,j_1+2\leq j_2,\ldots, j_{n-1}+2 \leq j_n} v_{j_1} ^2 v_{j_2} ^2\ldots v_{j_n} ^2 \nonumber   \\
        &&+t^2q^2\bigg\{ \sum_{4\leq j_1,j_1+2\leq j_2,\ldots, j_{n-1}+2 \leq j_n} v_{j_1} ^2 v_{j_2} ^2\ldots v_{j_n} ^2\bigg\}\nonumber \\&&+ q^4\bigg\{\sum_{5\leq j_1,j_1+2\leq j_2,\ldots, j_{n-1}+2 \leq j_n} v_{j_1} ^2 v_{j_2} ^2\ldots v_{j_n} ^2 \bigg\}\nonumber \\&&+s^2q^6\bigg\{\sum_{6\leq j_1,j_1+2\leq j_2,\ldots, j_{n-1}+2 \leq j_n} v_{j_1} ^2 v_{j_2} ^2\ldots v_{j_n} ^2 \bigg\} +\cdots \nonumber
        \end{eqnarray}
        Putting the values of $\{v_k\}$ as defined above, we get
        \begin{eqnarray}
        e_{n+1}&=&(q^{4n}+q^{12n+8}+\cdots)c_n+t^2(q^{6n+2}+q^{14n+10}+\cdots)d_n\nonumber \\
        &&+(q^{8n+4}+q^{16n+12}+\cdots)e_n+s^2(q^{10n+6}+q^{18n+14}+\cdots)h_n
    \nonumber\\ &=&\frac{q^{4n}}{1-q^{8n+8}}\big(c_n+q^{4n+4}e_n+t^2q^{2n+2}d_n+s^2q^{6n+6}h_n \big).\nonumber
 \end{eqnarray}
\end{proof}

\begin{theorem}\label{recurrence}
   Let $T=T(1,sq,q^2,tq^3,\ldots,q^{2n-2},sq^{2n-1},q^{2n},tq^{2n+1}, \ldots)$ be the weighted shift operator. Then the coefficients $\{c_n\}$ of the entire function $F_T(z)$ defined in \eqref{F_T(z)}, is given by 
    \begin{eqnarray*}
        c_1&=&\frac{1}{1-q^4}+ \frac{q^2(s^2+t^2q^4)}{1-q^8},\\
         c_2&=&\frac{q^{4}}{(1-q^{16})(1-q^4)}\bigg(1+q^8+q^6(s^2+t^2q^8)\\&&+\frac{t^2q^2(1+q^{12}+s^2q^2+s^2q^{10}+t^2q^{14})+s^2q^6(1+q^4+s^2q^2)}{1+q^4}\bigg)         \end{eqnarray*} and finally for $n\geq 1$, 
    \begin{eqnarray}
    c_{n+2}&=&\frac{q^{8n+4}}{1-q^{8n+16}}\bigg(t^2q^{2n+2}(1+s^2q^2)d_n+t^2q^{2n+12}d_{n+1}\\ &&\nonumber+c_n+q^4(1+q^4+s^2q^2)c_{n+1} \bigg),
\end{eqnarray}
     where the sequence $\{d_n\}$ is given in \eqref{dn+1}.
    \end{theorem}
    \begin{proof}
        Let the sequence $\{w_n\}$ denote the weight of the given weighted shift operator $T$. Then
        \begin{eqnarray*}
        w_{2k-1} &=& q^{2k-2} \mbox{ for } k\geq 1 \mbox{ and } w_{2k} = \begin{cases}
         sq^{2k-1}& \mbox{for } k=1,3,5,\ldots\\
        tq^{2k-1}&  \mbox{for } k=2,4,6,\ldots
    \end{cases}
        \end{eqnarray*}
        Then we have 
        \begin{eqnarray*}
        c_1=\sum_{k=1} ^\infty w_k ^2 = \frac{1}{1-q^4}+ \frac{q^2(s^2+t^2q^4)}{1-q^8} \mbox{ by Lemma \ref{HS}.}
        \end{eqnarray*}
        Now for $n\geq 1$, 
        \begin{eqnarray}
        c_{n+1}
        &=& \sum_{3\leq j_1,j_1+2\leq j_2,\ldots,j_{n-1}+2 \leq j_n} w_{j_1} ^2 w_{j_2} ^2\ldots w_{j_n} ^2 \nonumber   \\
        &&+s^2q^2\bigg\{ \sum_{4\leq j_1,j_1+2\leq j_2,\ldots, j_{n-1}+2 \leq j_n} w_{j_1} ^2 w_{j_2} ^2\ldots w_{j_n} ^2\bigg\}\nonumber \\&&+ q^4\bigg\{\sum_{5\leq j_1,j_1+2\leq j_2,\ldots, j_{n-1}+2 \leq j_n} w_{j_1} ^2 w_{j_2} ^2\ldots w_{j_n} ^2 \bigg\}\nonumber \\&&+t^2q^6\bigg\{\sum_{6\leq j_1,j_1+2\leq j_2,\ldots, j_{n-1}+2 \leq j_n} w_{j_1} ^2 w_{j_2} ^2\ldots w_{j_n} ^2 \bigg\} +\cdots \nonumber
        \end{eqnarray}
    Putting the values of $\{w_n\}$, we get
        \begin{eqnarray}
        c_{n+1} &=&
q^{4n}\bigg(1\times q^4 q^8\ldots q^{4(n-1)}+ \underbrace{(s^2\ldots s^2)(t^2\ldots t^2)}_{n \text{ terms}} q^2q^6q^{10}\ldots q^{4n-2}+ \cdots \bigg)\nonumber \\&&+ s^2q^2 q^{6n}\bigg(\underbrace{(t^2\ldots t^2)(s^2\ldots s^2)}_{n \text{ terms}} \times 1\times q^4q^8\ldots q^{4(n-1)}\nonumber \\&&+t^2 \times 1 \times q^6 q^{10} \ldots q^{4n-2}+\cdots \bigg)\nonumber \\&& +q^4 q^{8n}\bigg(1\times q^4 q^8\ldots q^{4(n-1)}+ \underbrace{(t^2\ldots t^2)(s^2\ldots s^2)}_{n \text{ terms}} q^6q^{10}\ldots q^{4n-2}+ \ldots \bigg)\nonumber \\&& +t^2q^6 q^{10n}\bigg(\underbrace{(s^2\ldots s^2)(t^2\ldots t^2)}_{n \text{ terms}} \times 1\times q^4q^8\ldots q^{4(n-1)}\nonumber \\&&+s^2 \times 1 \times q^6 q^{10} \ldots q^{4n-2}+\cdots \bigg)+\cdots 
\nonumber\\
        &=&
        (q^{4n}+q^{12n+8}+\cdots) e_n + s^2(q^{6n+2}+q^{14n+10}+\cdots) h_n \nonumber \\
        &&+ (q^{8n+4}+q^{16n+12}+\cdots) c_n + t^2 (q^{10n+6}+q^{18n+14}+ \cdots) d_n
        \nonumber\\
        &=& 
        \frac{q^{4n}}{1-q^{8n+8}}\big(e_n+q^{4n+4}c_n+s^2q^{2n+2}h_n+t^2q^{6n+6}d_n \big)\label{Cn+1}
         \end{eqnarray} 
where $\{e_n\}$ and $\{h_n\}$ are defined in \eqref{en+1} and  \eqref{hn+1} respectively. Replacing $n$ by $n+1$, we get
         \begin{eqnarray} \label{cn+2}
             c_{n+2}&=&\frac{q^{4n+4}}{1-q^{8n+16}}\big(e_{n+1}+q^{4n+8}c_{n+1}+s^2q^{2n+4}h_{n+1}\nonumber\\&&+t^2q^{6n+12}d_{n+1} \big).
         \end{eqnarray}
Now we have,
\begin{eqnarray*}
    d_{n+1} &=&-q^{2n-2}e_n+q^{-2n-2}c_{n+1}, \,\,\,\,\, \mbox{[by \eqref{dn+1} and \eqref{Cn+1}]}\\
     h_{n+1} &=& q^{4n}(t^2d_n+q^{2-2n}c_{n+1}) \,\,\,\,\, \mbox{[by \eqref{hn+1} and \eqref{Cn+1}]} \,\,\,\\
 \mbox{and} \,\,\,\,\,\,\,  e_{n+1} &=& q^{4n}(t^2q^{2n+2}d_n+q^4c_{n+1}+c_n) \,\,\,\,\, \mbox{[by \eqref{en+1} and \eqref{Cn+1}]}.
\end{eqnarray*}
Hence finally putting the values of $h_{n+1}$ and $e_{n+1}$ in \eqref{cn+2},
\begin{eqnarray*}
    c_{n+2}&=&\frac{q^{8n+4}}{1-q^{8n+16}}\bigg(t^2q^{2n+2}(1+s^2q^2)d_n+t^2q^{2n+12}d_{n+1}\\ &&+c_n+q^4(1+q^4+s^2q^2)c_{n+1} \bigg).
\end{eqnarray*}
 \end{proof}
 \begin{remark}
     Putting $s=t$ in Theorem \ref{recurrence}, we get
     \begin{eqnarray*}
        c_1 &=& \frac{1+s^2q^2}{1-q^4},\\
             c_2 &=& \frac{q^{4}}{(1-q^{16})(1-q^4)}\bigg(1+q^8+q^6(s^2+s^2q^8)\\&&+\frac{s^2q^2(1+q^{12}+s^2q^2+s^2q^{10}+s^2q^{14})+s^2q^6(1+q^4+s^2q^2)}{1+q^4}\bigg)\\ 
             &=& \frac{q^{4}}{(1-q^{16})(1-q^4)}\bigg((1+q^8)(1+s^2q^6)\\&&+\frac{s^2q^2(1+q^{12}+s^2q^2+s^2q^{10}+s^2q^{14}+q^4+q^8+s^2q^6)}{1+q^4}\bigg)\\
             &=& \frac{q^{4}}{(1-q^{16})(1-q^4)}\bigg((1+q^8)(1+s^2q^6)\\&&+\frac{s^2q^2(1+q^8)(1+q^4+s^2q^2+s^2q^6)}{1+q^4}\bigg)\\
             &=& \frac{q^{4}}{(1-q^8)(1-q^4)}\bigg((1+s^2q^6)+\frac{s^2q^2(1+q^4+s^2q^2(1+q^4))}{1+q^4}\bigg)
             \\&=&\frac{q^4(1+s^2q^2+s^4q^4+s^2q^6)}{(1-q^8)(1-q^4)}
        \end{eqnarray*} and 
        \begin{eqnarray}\label{Cnn}
        c_{n+2} &=&\nonumber\frac{q^{8n+4}}{1-q^{8n+16}}\bigg(s^2q^{2n+2}(1+s^2q^2)d_n+s^2q^{2n+12}d_{n+1}\nonumber\\ 
        &&+c_n+q^4(1+q^4+s^2q^2)c_{n+1} \bigg).
        \end{eqnarray}
        For \(s=t\), the operators \(T_2\) and \(T_3\) defined in \eqref{T_2} and \eqref{T_3} coincide; likewise, the operators \(T_4\) and \(T\) defined in \eqref{T_4} and \eqref{Main_operator} are identical. Hence from \eqref{d_k,h_k,e_k}, we have
        \begin{eqnarray*}
            d_n=h_n,\   c_n=e_n 
        \end{eqnarray*}
        whenever $s=t$.
  Therefore from \eqref{dn+1} and \eqref{en+1}, we have \begin{eqnarray}\label{c_n,d_n}
            d_{n+1} &=& \frac{q^{4n}}{1-q^{4n+4}}(s^2d_n+q^{2n+2}c_n)\label{s=t1}\\
            c_{n+1} &=& \frac{q^{4n}}{1-q^{4n+4}}(c_n+s^2q^{2n+2}d_n)\label{s=t}
        \end{eqnarray}
for all $n\geq 1$. It follows from \eqref{s=t} that
\begin{eqnarray*}
    d_n &=&\frac{1}{s^2q^{2n+2}}\Big(\frac{1-q^{4n+4}}{q^{4n}}c_{n+1}-c_n\Big)
    \end{eqnarray*}
and hence from \eqref{s=t1}, we obtain
   \begin{eqnarray*} 
   d_{n+1} &=&
   \frac{q^{4n}}{1-q^{4n+4}}\bigg(\frac{1}{q^{2n+2}}\left(\frac{1-q^{4n+4}}{q^{4n}}c_{n+1}-c_n\right)+q^{2n+2}c_n\bigg)
   \\&=&
   q^{4n}\Big(\frac{1}{q^{6n+2}}c_{n+1}-\frac{1}{q^{2n+2}}c_n\Big).
\end{eqnarray*}
Substituting the values of $d_n$ and $d_{n+1}$ in \eqref{Cnn} we have, 
         \begin{eqnarray*}
            c_{n+2}&=&\frac{q^{8n+4}}{1-q^{8n+16}}\bigg((1+s^2q^2)\left(\frac{1-q^{4n+4}}{q^{4n}}c_{n+1}-c_n\right)\\&&+s^2q^{6n+12}\left(\frac{1}{q^{6n+2}}c_{n+1}-\frac{1}{q^{2n+2}}c_n\right)+c_n\\ 
        &&+q^4(1+q^4+s^2q^2)c_{n+1} \bigg)  \\
&=&\frac{q^{8n+4}}{1-q^{4n+8}}\bigg(-s^2q^2c_n+\frac{1+s^2q^2}{q^{4n}}c_{n+1}\bigg)
           \\&=&\frac{q^{4n+4}(1+s^2q^2)}{1-q^{4n+8}}c_{n+1}-\frac{s^2q^{8n+6}}{1-q^{4n+8}}c_n.
        \end{eqnarray*}
       The above values for $c_1, c_2$ and $c_{n+2}$ coincide with the expressions previously obtained in Theorem 3.1 of \cite{chakraborty2025numerical}.
 \end{remark}
From Stout\cite{stout1983numerical}, we have $w(T(w_1,w_2,\ldots ,w_{n-1}))$ converges to \\
$w(T(w_1,w_2,\ldots))$ as $n \rightarrow{\infty}$, where the weights $(w_1,w_2,\ldots)$ are positive and square summable. Here we have the following MATLAB program similar to  \cite{chakraborty2025numerical} which gives the numerical radii of weighted shift matrices with weights

$(1,sq,q^2,tq^3,\ldots,q^{4n})$, 
$(1,sq,q^2,tq^3,\ldots,sq^{4n+1})$,
$(1,sq,q^2,tq^3,\ldots,q^{4n+2})$  and 
$(1,sq,q^2,tq^3,\ldots,tq^{4n+3})$   where $s,t >0$ and $0<q<1$ and $n\geq 1$.\\
\begin{program}
\begin{lstlisting}[style=Matlab-editor]

format long;
input s;
input q;
input t;
input N(N>=6); 
lambda_max= zeros(1,N);
for n=6:N
w = zeros(1, n);
for k = 1:n-1
    if mod(k,4)==1
        w(k) = q^(k-1);
    elseif mod(k,4)==2     
        w(k) = s*q^(k-1);
    elseif mod(k,4)==3      
        w(k) = q^(k-1);
    elseif mod(k,4)==0      
        w(k) = t*q^(k-1);
    end
end
A = zeros(n);
for i = 1:n-1
    A(i+1,i) = w(i);
end
H = (A + A')/2;
lambda_max(n) = max(eig(H));
end
disp('The weighted shift matrix A =');
disp(A);
fprintf ('w(A) = %.16f', lambda_max(n));
figure;
stem(6:N,lambda_max(6:N),'LineWidth', 2);
xlabel('n');
ylabel('max eigenvalue of H');
grid on;
\end{lstlisting}\label{program1}
\end{program}
For a weighted shift operator $T=T(w_1,w_2,\ldots)$ with square summable positive weights $(w_1,w_2,\ldots)$, Stout\cite{stout1983numerical} has proved that  $w(T)=1/\lambda$, $\lambda$ being the smallest positive root of $F_T(z)=0$. We consider the partial sum of $F_T(z)$ as 
    \begin{eqnarray}\label{F_n}
        F_n(z)=1+\sum_{k=1} ^n \bigg(\frac{-1}{4}\bigg)^kc_kz^{2k}
    \end{eqnarray}
    for all $n\geq 1$. We know that $F_n(z)$ converges to $F_T(z)$ uniformly on compact sets. Moreover, as shown in \cite{chakraborty2025numerical}, there exists a subsequence $\{z_{n_{k}}\}$ of zeros of $F_{n_{k}}(z)=0$ which converges to $\lambda$.

Now for the weighted shift operator $T=T(1,sq,q^2,tq^3,\ldots )$ with $s,t>0, 0<q<1$ we have the following MATLAB program similar to \cite{chakraborty2025numerical} to find the roots of the sequence of polynomials (\ref{F_n}).\\

\begin{program}
\begin{lstlisting}[style=Matlab-editor]

format long;
digits(100);
q = input;
t = input;
s = input;
m = input;
c = sym(zeros(1, m+2));
d = sym(zeros(1, m+1));
e = sym(zeros(1, m+1));
c(1) = vpa((1/(1-q^4))+q^2*(s^2 + t^2*q^4)/(1-q^8));
d(1) = vpa((s^2 + t^2*q^4)/(1-q^8)+q^2/(1-q^4));
e(1) = vpa((1/(1-q^4))+q^2*(t^2 + s^2*q^4)/(1-q^8));
c(2) = vpa(q^4/((1-q^16)*(1-q^4))*(1+q^8+q^6*(s^2+t^2*q^8)+ 
(t^2*q^2*(1+q^12+s^2*q^2+ s^2*q^10+t^2*q^14)+s^2*q^6*(1+q^4+s^2*q^2))/(1+q^4)));
for i = 1:m
e(i+1) = vpa(q^(4*i)*(t^2*q^(2*i+2)*d(i)+q^4*c(i+1)+c(i)));
d(i+1) = vpa(-q^(2*i-2)*e(i)+q^(-2*i-2)*c(i+1));
c(i+2) = vpa((q^(8*i+4)/(1-q^(8*i+16)))*(t^2*q^(2*i+2)*(1+s^2*q^2)*d(i)+(t^2*q^(2*i+12)*d(i+1)+c(i)+q^4*(1+q^4+s^2*q^2)*c(i+1))));
end
syms z;
F2 = sym(1);
for k = 1:m
F2=F2+vpa((-1/4)^k*c(k))*z^(2*k);
end
F2 = expand(F2);
disp('polynomial F2(z):');
pretty(F2);
sol=vpa(solve(F2 == 0, z), 80);
disp('roots of F2:');
disp(sol);
    \end{lstlisting}\label{program2}
\end{program}
\begin{example}
    Let $T=T(1,sq,q^2,tq^3,\ldots)$ be the weighted shift operator with $q=0.08$ and $s=0.1$ and $t=0.01$. From Theorem \ref{ag}, the upper bound and lower bound of numerical radius are
\begin{align*}
    \frac{1}{2}\max\{1,sq,tq^3\}+\frac{\sqrt{sq}}{2}\max\{1,\sqrt{\frac{t}{s}}q^2\} =0.54472\,\,
    \mbox{and }\\ \ \sup\limits_{0<z<1} \displaystyle\frac{(1-z)\sqrt z}{(1-q^4z^4)} \big((1+sqz)+q^2z^2(1+tqz)\big),\,\, \mbox{ which is  approximately $0.38454519$},
\end{align*}
respectively.\\
Also from \cite{stout1983numerical}, we already know that the numerical radius of the weighted shift matrix $T_n$ defined as
\begin{center}
$T_n=$
$\begin{cases}
T(1,sq,q^2,tq^3,\ldots ,q^{n-1})& \mbox{ if $n=4k+1$,}\\
T(1,sq,q^2,tq^3,\ldots ,sq^{n-1})& \mbox{ if $n=4k+2$,}\\
T(1,sq,q^2,tq^3,\ldots ,q^{n-1})& \mbox{ if $n=4k+3$ and}\\
T(1,sq,q^2,tq^3,\ldots ,tq^{n-1})& \mbox{ if $n=4k$},\,\,\,\,\, \mbox{where $k \in \mathbb{N}$}
\end{cases}$
\end{center}
converges to $w(T)$ as $n\to\infty$ where $T$ is defined in \eqref{Main_operator}. The numerical computations by Program \ref{program1} indicate that the largest eigenvalue of $\mbox{Re}(T_{n})$ converges toward a limiting value  $w(T)$ which is approximately $0.500016$. This empirical behaviour is clearly reflected as shown in Figure~\ref{fig:mesh4}.\\ 
    Also from Remark 3.3 of \cite{chakraborty2025numerical}, we know that there exists a subsequence $\{z_{n_k}\}$ of zeros of $F_{n_k}(z)$ which converges to $\lambda=\frac{1}{w(T)}$. Here, the approximate minimal positive roots of $F_1(z)=0$, $F_2(z)=0$, $F_3(z)=0$ are $1.999895046$, $1.99993600045$, $1.99993600045$ respectively. The numerical computations by Program \ref{program2} show that, for large $n$, the minimal positive root of $F_n(z)=0$ appears to converge to approximately \(1.99993600045\), whose reciprocal is \(0.500016\), in agreement with the data presented in Figure~\ref{fig:mesh4}. Consequently, we may regard \(w(T)\) as being approximately \(0.500016\).

\begin{figure}[h]
    \centering
\includegraphics[width=0.52\textwidth]{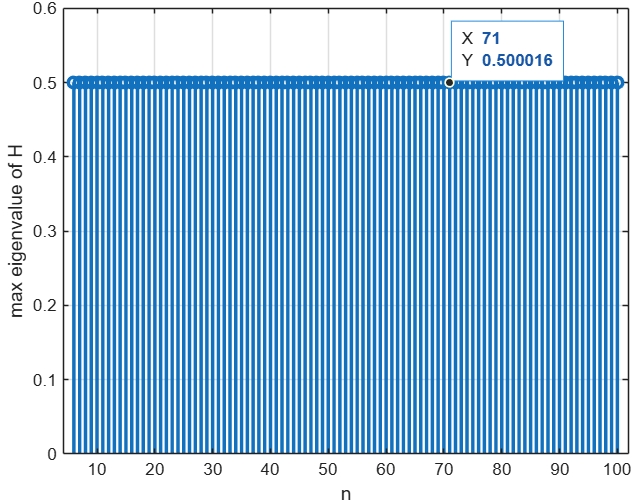}
    \caption{}
    \label{fig:mesh4}
\end{figure}

\end{example}
\begin{note}
 All numerical calculations have been done in this article using  Matlab.
\end{note}
\section{Declarations}
\textit{Acknowledgement:} Mr. Arobinda Ghosh would like to thank UGC, Govt. of India for the financial support (NTA Ref. No. 211610122278) in the form of fellowship.\\
\textit{Author Contributions:} All the authors contributed equally to this manuscript and reviewed it. \\
 \textit{Data Availability :} No datasets were generated or analysed during the current study. \\
\textit{Conflict of interest:} The authors declare no conflict of interest.\\
\textit{Competing interest:} The authors declare no competing interests.
\bibliographystyle{amsplain}

 \end{document}